\documentclass{amsart}
\newtheorem{thr}{Theorem}[section]
\newtheorem{lem}[thr]{Lemma}
\newtheorem{cor}[thr]{Corollary}

\theoremstyle{definition}
\newtheorem{defn}[thr]{Definition}

\theoremstyle{remark}

\numberwithin{equation}{section}

\def\varkappa{\kappa}
\def\Ro{\overline{\mathbb{R}}}
\def\R{\mathbb{R}}

\def\P{{\mathcal P}}
\def\G{{\mathcal G}}
\def\S{{\mathcal S}}
\def\op{\oplus}
\def\ot{\otimes}
\def\R{\mathbb{R}}

\begin{document}

\title{Tropical matrices and group representations}

\author{Yaroslav Shitov}
\address{Moscow State University, Leninskie Gory, 119991, GSP-1, Moscow, Russia}
\email{yaroslav-shitov@yandex.ru}


\begin{abstract}
The paper gives a complete description of the subgroups of the semigroup of tropical $n$-by-$n$ matrices
up to an isomorphism. In particular, we show that every of these groups has a torsion-free
abelian subgroup of index at most $n!$, proving the conjecture of Johnson and Kambites.
\end{abstract}

\maketitle

\section{Introduction}

The set $\R$ of reals extended by adding an infinite negative element $-\infty$ is called
the \textit{tropical semiring} and is also known as the \textit{max-plus algebra}.
The tropical arithmetic operations on $\Ro=\R\cup\{-\infty\}$
are $a\op b=\max\{a,b\}$ and $a\ot b=a+b$. The main object of our research is the set
$\Ro^{n\times n}$ of tropical $n$-by-$n$ matrices. We are interested in studying
the multiplicative structure of tropical matrices. The multiplication of such matrices
is defined as ordinary matrix multiplication with $+$ and $\cdot$ replaced by the tropical
operations $\op$ and $\ot$.

The study of linear algebra over the tropical semiring is important for many
different applications (see~\cite{AGG, HOW}). There is a number of purely linear-algebraic
important problems for tropical matrices, for example, the eigenvalues and eigenvectors
problem, the problem of solving linear systems, computational problems for the rank functions,
see~\cite{DSS, HOW}. Another important approach considers the set of tropical
matrices from the point of view of the semigroup theory. The paper~\cite{Gau} is
devoted to the solution of the Burnside-type problem for semigroups of tropical matrices.
Johnson and Kambites in the recent paper~\cite{JK2} have developed the study of the
semigroup-theoretic structure of tropical matrices. They consider Green's relations
on the semigroup $\left(\Ro^{n\times n},\ot\right)$, groups of tropical matrices,
and idempotent tropical matrices. They give a complete description of the subgroups of
$\left(\Ro^{n\times n},\ot\right)$ in the case when $n=2$. The study of Green's relations
on the semigroup of tropical matrices has been developed in~\cite{HK, JK}. In~\cite{HK},
the complete description of the $\mathcal{D}$-relation has been provided. In~\cite{JK},
the important characterization of the $\mathcal{J}$-order has been given, and the connection
of Green's relations with the rank functions of tropical matrices has been studied.

The aim of our paper is to solve the problem that has arisen from the paper~\cite{JK2}.
Namely, we are interested in a complete characterization of the subgroups of the
semigroup $\left(\Ro^{n\times n},\ot\right)$. We show that every subgroup of the
semigroup of tropical $n$-by-$n$ matrices admits a faithful representation with
tropical monomial $n$-by-$n$ matrices. We prove that every subgroup $\G$ of
$\left(\Ro^{n\times n},\ot\right)$ is isomorphic to a subgroup of the wreath product
$\R\wr\S_n$, and, conversely, every subgroup of $\R\wr\S_n$ can be realized with
tropical $n$-by-$n$ matrices. Our results confirm the conjecture proposed
in~\cite{JK2} that every group admitting a faithful representation by $n\times n$
tropical matrices must have a torsion-free abelian subgroup of index at most $n!$.
We also give an upper bound for the order of a periodic group of tropical
$n$-by-$n$ matrices, developing the result proven in~\cite{dAP}.

Throughout our paper $\S_n$ will denote the symmetric group on $\{1,\ldots,n\}$.
By $a_{i(\cdot)}$ we denote the $i$th row of a matrix $A$, by $A[r_1,\ldots,r_k]$ the submatrix of $A$
formed by the rows with numbers $r_1,\ldots,r_k$.
We say that a matrix $P\in\Ro^{n\times n}$ is \textit{monomial} if there
exists $\sigma=\sigma(P)\in\S_n$ such that $p_{ij}\neq-\infty$ if and only if $i=\sigma(j)$.
In this case, $P$ is called \textit{diagonal} if $\sigma(P)$ is the identity.
Note that the diagonal matrix with zeros on the diagonal is the neutral element of
the semigroup $\left(\Ro^{n\times n},\ot\right)$.

\section{Subgroups of the semigroup $\left(\Ro^{n\times n},\ot\right)$}

We need the concept of the row rank (see~\cite{AGG}) of a tropical matrix for our considerations.

\begin{defn}\label{rrank}
A tropical matrix $B\in\Ro^{n\times m}$ is said to be of \textit{full row rank} if no row
of $A$ can be expressed as a linear combination of other rows, that is, the condition
\begin{equation}\label{eq_1_of_it}
b_{i(\cdot)}=\bigoplus_{k\in\{1,\ldots,n\}\setminus\{i\}} \lambda_k\ot b_{k(\cdot)}
\end{equation}
fails to hold for every $i\in\{1,\ldots,n\}$ and $\lambda_1,\ldots,\lambda_n\in\Ro$.
\end{defn}

We also need the following lemma.

\begin{lem}\label{lemma_1_of_it}
Let $A, B, C, D\in\Ro^{n\times n}$ be such that $B=C\ot A$, $A=D\ot B$.
If the row rank of $B$ is full, then there exists a monomial matrix $P\in\Ro^{n\times n}$
such that $B=P\ot A$.
\end{lem}

\begin{proof}
Since $B=C\ot A$ and $A=D\ot B$ imply that $B=C\ot D\ot B$, we have
\begin{equation}\label{eq_2 of_it}
b_{i(\cdot)}=\bigoplus\limits_{p=1}^n\left[\left(\bigoplus\limits_{k=1}^nc_{ik}\ot d_{kp}\right)\ot b_{p(\cdot)}\right].
\end{equation}


If $\bigoplus\limits_{k=1}^nc_{ik}\ot d_{ki}<0$, then
the summand of $p=i$ can be omitted from the right-hand side of~(\ref{eq_2 of_it}). In this case,
the condition~(\ref{eq_1_of_it}) is satisfied, so the row rank of $B$ is not full. The contradiction
shows that indeed $\bigoplus_{k=1}^n\left(c_{ik}\ot d_{ki}\right)\geq0$.

Then for some $\varkappa=\varkappa(i)$, we have $c_{i\varkappa}\ot d_{\varkappa i}\geq0$.
From $B=C\ot A$ it follows that $b_{it}\geq c_{i\varkappa}\ot a_{\varkappa t}$ for every $t\in\{1,\ldots,n\}$.
On the other hand,
$$c_{i\varkappa}\ot a_{\varkappa t}=c_{i\varkappa}\ot\left(\bigoplus\limits_{p=1}^n d_{\varkappa p}\ot b_{pt}\right)
\geq c_{i\varkappa}\ot d_{\varkappa i}\ot b_{it}\geq b_{it},$$
so indeed $b_{it}=c_{i\varkappa}\ot a_{\varkappa t}$ for every $t\in\{1,\ldots,n\}$.
Thus every row of $B$ is some row of $A$ multiplied by a scalar.
Since the row rank of $B$ is full, the result follows.
\end{proof}

The following lemma deals with matrices whose row rank is not full.

\begin{lem}\label{lemma_2_of_it}
Let $\G$ be a subgroup of $\left(\Ro^{n\times n},\ot\right)$, $n\geq2$.
If the row rank of some matrix $A$ from $\G$ is not full, then $\G$ admits
a faithful representation with tropical $(n-1)$-by-$(n-1)$ matrices.
\end{lem}

\begin{proof}
By Definition~\ref{rrank}, some row of $A$ is a linear combination of other its rows.
So for some $i\in\{1,\ldots,n\}$ we have $A=P\ot\overline{A}$, where the matrix
$\overline{A}\in\Ro^{(n-1)\times n}$ is obtained
from $A$ by removing the $i$th row, and $P\in\Ro^{n\times(n-1)}$ is such that the matrix
$P[1,\ldots,i-1,i+1,\ldots,n]$ is the neutral element of $\left(\Ro^{(n-1)\times(n-1)},\ot\right)$.
Since $\G$ is a group, for every $G\in\G$ there exists $B\in\G$ such that $G=A\ot B$.

Thus we see that $G=P\ot\overline{G}$. The matrix $\overline{G}$ here is uniquely
determined, because $P[1,\ldots,i-1,i+1,\ldots,n]$ is the neutral element.
The map $\varphi$ sending $G\in\G$ to
$\overline{G}\ot P\in\Ro^{(n-1)\times(n-1)}$ is therefore well defined.

Note that for every $G,H\in\G$ it holds that
$$\varphi(G\ot H)=\varphi(P\ot \overline{G}\ot P\ot \overline{H})=\overline{G}\ot P\ot \overline{H}\ot P=\varphi(G)\ot \varphi(H),$$
so $\varphi$ is a homomorphism. Moreover, if $\varphi(G)=\varphi(H)$, then $\overline{G}\ot P=\overline{H}\ot P$,
so in this case $P\ot \overline{G}\ot P\ot \overline{G}=P\ot \overline{H}\ot P\ot \overline{G}$, or $G\ot G=H\ot G$.
Since $\G$ is a group, the condition $\varphi(G)=\varphi(H)$ therefore implies that $G=H$, proving that $\varphi$ is injective.
\end{proof}

Now we are ready to prove the one of our main results.

\begin{thr}\label{grtrmatr}
Every subgroup of the semigroup $\left(\Ro^{n\times n},\ot\right)$ admits
a faithful representation with tropical monomial $n$-by-$n$ matrices.
\end{thr}

\begin{proof}
The case of $n=1$ is trivial, and we proceed by the induction on $n$.
Let $n\geq2$, $\G$ be a subgroup of $\left(\Ro^{n\times n},\ot\right)$,
$E$ be a neutral element of $\G$. The two cases are then possible.

1. Let $\G$ contain a matrix whose row rank is not full.
Lemma~\ref{lemma_2_of_it} shows that in this case $\G$
admits a faithful representation with tropical $(n-1)$-by-$(n-1)$
matrices. The inductive hypothesis then shows that $\G$ has
a faithful representation with tropical monomial $(n-1)$-by-$(n-1)$
matrices, so the result follows.

2. Now let the matrices from $\G$ be of full row rank.
In this case, from Lemma~\ref{lemma_1_of_it} it follows that
for every $G\in\G$ there exists a monomial matrix $\P_G\in\Ro^{n\times n}$
such that $G=\P_G\ot E$. Since the row rank of $G$ is full, we see that
the matrix $\P_G$ is uniquely determined. So we can define the map
$\psi$ sending $G\in\G$ to the monomial matrix $\P_G$. Clearly,
$\psi$ is injective. We also see that for every $G,H\in\G$ it holds that
$G\ot H=\P_G\ot E\ot H=\P_G\ot H=\P_G\ot\P_H\ot E,$
so $\psi$ is a homomorphism.
\end{proof}

Johnson and Kambites in~\cite[Section 4]{JK2} conjectured that every group admitting a faithful
representation by $n\times n$ tropical matrices has a torsion-free abelian subgroup of index at most $n!$.
Now we are ready to prove this conjecture.

\begin{thr}\label{answ}
Let a group $\G$ admit a faithful representation by $n\times n$ tropical matrices.
Then $\G$ has a torsion-free abelian subgroup of index at most $n!$.
\end{thr}

\begin{proof}
By Theorem~\ref{grtrmatr}, we assume without a loss of generality that $\G$
consists of tropical monomial $n$-by-$n$ matrices. Consider the subgroup $D$
of all diagonal matrices from $\G$. Clearly, $D$ is normal in $\G$,
abelian and torsion-free. It remains to note that matrices $A,B\in\G$ belong
to the same coset of $D$ in $\G$ if and only if $\sigma(A)=\sigma(B)$.
\end{proof}

D'Alessandro and Pasku have shown that every periodic finitely generated subgroup
of $\left(\Ro^{n\times n},\ot\right)$ is finite, see~\cite[Proposition 5]{dAP}. Theorem~\ref{grtrmatr}
allows us to derive a more precise characterization. Recall that a group $H$ is
\textit{periodic} if each element of $H$ has finite order.

\begin{cor}\label{answ2}
The order of any periodic subgroup of the semigroup $\left(\Ro^{n\times n},\ot\right)$ is at most $n!$.
\end{cor}

\begin{proof}
By definition, any torsion-free subgroup of a periodic group is trivial. So the result
follows from Theorem~\ref{answ}.
\end{proof}

Finally, we note that the group of all tropical monomial $n$-by-$n$ matrices is isomorphic
to the wreath product $\R\wr\S_n$. This gives the following group-theoretic description of
the subgroups of $\left(\Ro^{n\times n},\ot\right)$.

\begin{thr}\label{grtrmatr2}
A group $\G$ admits a faithful representation with tropical $n$-by-$n$ matrices if
and only if $\G$ is isomorphic to a subgroup of the wreath product $\R\wr\S_n$.
\end{thr}

\begin{proof}
Follows from Theorem~\ref{grtrmatr}.
\end{proof}

I would like to thank my scientific advisor Professor Alexander E. Guterman for constant attention to my work.

\end{document}